\documentclass[11pt,a4paper,onecolumn]{article} 

\usepackage{latexsym,amssymb,amsmath, amsbsy,amsopn, amstext}
\usepackage{graphicx,epsfig,tikz,pgf,hyperref,cancel,color,float,ifpdf}
\usepackage{extarrows,mathrsfs}
\usepackage{amsmath,amssymb,amsthm,stfloats,subcaption,indentfirst}

\graphicspath{{./}}
\newcommand{\R}{\mathbb{R}}
\newcommand{\Z}{\mathbb{Z}}

\newtheorem{definition}{\bfseries Definition}
\newtheorem{example}{\bfseries Example}
\newtheorem{theorem}{\bfseries Theorem}

\newtheorem{lemma}{\bfseries Lemma}

\newtheorem{remark}{\bfseries Remark}

\allowdisplaybreaks[1]
\title{Periodic Event-triggered Control for Incrementally Quadratic
	Nonlinear Systems}
\date{}
\author{Xiangru Xu\footnote{X. Xu is with Department of Mechanical Engineering,  University of Wisconsin-Madison, Madison, WI, USA. Email: {\tt\small xiangru.xu@wisc.edu}.}, Adam M. Tahir, Beh\c{c}et A\c{c}\i kme\c{s}e\footnote{A. Tahir and B. A\c{c}\i kme\c{s}e are with  Department of Aeronautics \& Astronautics, University of Washington, Seattle, WA, USA. Emails: {\tt\small \{tahiram,behcet\}@uw.edu}.}}

\begin{document}
\maketitle

\begin{abstract}   
{\color{black}Periodic event-triggered control (PETC) evaluates the triggering rule periodically and is well-suited for implementation on digital platforms. 
This paper investigates PETC design for nonlinear systems affected by external disturbances under the impulsive system formulation.} Sufficient conditions are provided to ensure the input-to-state stability of the resulting closed-loop system for the state feedback  and the observer-based output feedback configurations separately. For each configuration, the sampling period and the triggering functions are provided explicitly. Sufficient conditions in the form of linear matrix inequalities are provided for the PETC design of incrementally quadratic nonlinear systems. Two examples are given to illustrate the effectiveness of the proposed method.
\end{abstract}


\section{Introduction}\label{sec:intro}
Digital control systems are  traditionally executed in a time-triggered fashion where the sensors and actuators are accessed periodically.  
In contrast, event-triggered control (ETC) executes the communication and actuation only when certain triggering rules are satisfied; this can be seen as adding feedback to the communication and actuation processes (see a recent survey paper \cite{heemels2012introduction} and references therein). The ETC paradigm is designed to avoid unnecessary waste of communication/computation resources by reducing the number of communication/actuation executions, while still guaranteeing a desirable closed-loop performance  \cite{tabuada2007event,shoukry2016event,wang2011event,garcia2013model,borgers2014event,donkers2012output,abdelrahim2017robust,abdelrahim2016stabilization}; this shows potential in applications of systems with limited communication bandwidth such as networked control systems.

Since the triggering condition of ETC has to be monitored continuously, it is difficult to implement ETC in digital platforms directly. 
{\color{black}By evaluating the triggering conditions and deciding whether to  update the communication/actuation at each periodic sampling time, periodic event-triggered control (PETC) inherits advantages of ETC and can be implemented on standard digital platforms \cite{heemels2013periodic,heemels2013model,heemels2015periodic}. } Furthermore, Zeno phenomenon is avoided since the sampling period is a lower bound for the minimum inter-execution time. Although ETC for discrete-time models can be considered as PETC (e.g., see \cite{heemels2013model,eqtami2010event}), the inter-sample behavior of the original continuous-time systems are not captured in the discrete-time analysis. 
PETC design for (continuous-time) linear systems was investigated in \cite{heemels2013periodic,linsenmayer2019periodic}. 
However, PETC design for nonlinear systems is difficult because of an intrinsic difficulty: the discrete-time dynamics of a nonlinear system can not be exactly known from its continuous-time dynamics  \cite{luc2017periodic,borgers2018periodic,wang2016stabilization,aranda2017design,yang2018periodic}. There are few existing papers on PETC design for nonlinear systems: \cite{wang2016stabilization} studied state feedback PETC design for undisturbed nonlinear systems using the hybrid system approach and proved globally asymptotically stability of the closed-loop system; \cite{wang2018periodic} studied output feedback PETC design for disturbed nonlinear systems; \cite{wang2019periodic} extended the results of \cite{wang2016stabilization,wang2018periodic} to the decentralized setting; \cite{borgers2018periodic}  investigated state feedback PETC design for nonlinear systems by redesigning the event function of an existing continuous ETC system using overapproximation techniques, such that control performance guarantees for the continuous ETC system are preserved; \cite{luc2017periodic} studied output feedback PETC design for Lipschitz systems using impulsive observers and proved practical stabilization of the resulting system.  In spite of these interesting results, many PETC design problems for nonlinear systems are largely open and deserve to be further explored. 

{\color{black}This paper investigates input-to-state stabilization of disturbed nonlinear systems under PETC mechanisms. 
An impulsive system approach is used for modeling and analyzing the overall system. Under the assumption that there exists a sum-type ISS-Lyapunov function for the continuous dynamics, the sampling period and triggering functions are designed such that the overall system is ISS using the same ISS-Lyapunov function. The contributions of the paper are at least two-fold: (i) This work provides sufficient conditions for the input-to-state stabilization of nonlinear systems affected by disturbances using state feedback or observer-based output feedback PETC design. 
(ii) This work presents sufficient conditions in the form of linear matrix inequalities (LMIs) for the PETC design of incrementally quadratic nonlinear systems,  which subsumes many class of nonlinear systems including the Lipschitz nonlinear systems.  
Degeneration of the general results to linear control systems is also discussed.  Compared with \cite{wang2019periodic,wang2016stabilization,wang2018periodic}, this work does not rely on the small-gain techniques for hybrid systems, and provides LMI conditions for more general class of systems; compared with \cite{borgers2018periodic,luc2017periodic}, the work considers general nonlinear dynamics with external disturbances. }
\emph{Notation.} 
Denote the set of real, non-negative real and non-negative integer numbers  by $\R$, $\R_{\geq 0}$ and $\Z_{\geq 0}$, respectively. 
Denote the 2-norm by $\|\cdot\|$. Given a non-empty and closed set $\mathcal{A}$, the point-to-set distance from $x$ to $\mathcal{A}$ is denoted by $\|x\|_{\mathcal{A}}=\inf_{y\in\mathcal{A}}\|y-x\|$. Denote the identity matrix of size $n\times n$ by $I_n$. Denote the zero matrix of size  $n_1\times n_2$
by ${\bf 0}_{n_1\times n_2}$ and the zero vector of size $n$
by ${\bf 0}_{n}$; the subscripts  will be omitted when  clear from context. 
Denote the block diagonal matrix by $diag\{M_1,...,M_n\}$ where $M_1,...,M_n$ are matrices in the diagonal block. 
For symmetric matrices, $*$ stands for entries whose values follow from symmetry. A signal $x: \R_{\geq 0}\rightarrow \R^n$ is called left-continuous if $\lim_{s\rightarrow t^-}x(s)=x(t)$ for all $t>0$. ``$\forall x\;a.e.$'' means for every $x\in\R^{n_x}$ except for a set of zero Lebesgue-measure in $\R^{n_x}$. The definitions of $\mathcal{K}$-function, $\mathcal{K}_{\infty}$-function, $\mathcal{KL}$-function, and input-to-state stability can be found in Section 4.4 of \cite{khalil02book}.

\section{Problem Statement}\label{sec:formulation}

Fig.\ref{figstate} (a) shows the configuration of implementing the state feedback PETC. 
The plant is a nonlinear system given as
\begin{align}
\dot{x}(t)&=f(x(t),u(t),w(t))\label{NLsys1}
\end{align}
where $x\in\R^{n_x}$ is the state, $u\in\R^{n_u}$ is the control input, $w\in\R^{n_w}$ is the disturbance, $f$ is a locally Lipschitz continuous function. The state feedback controller is given as $u(t)=k(x(t))$ where $k(\cdot)$ is a continuous function. Assume $u(t)$ is designed such that the solution to  the system $\dot{x}(t)=f(x(t),k(x(t)),w(t))$ exists for all time and all initial conditions, and the closed-loop system is input-to-state stable (ISS) with respect to (w.r.t.) $w$.

Denote the sampling period to be $h>0$, and define the sampling times as $t_k:=kh$ for any $k\in\Z_{\geq 0}$. With the event-triggering mechanism (ETM), the state of the plant, $x(t)$, is sampled 
at each sampling time $t_k$. The input to the controller, $\tilde x_c(t)$, is updated only when the event-triggering condition for the state is satisfied. Specifically, $\tilde x_c(t)$ is a left-continuous, piecewise constant signal that is  defined for $t\in(t_k,t_{k+1}]$ as 
\begin{align}\label{stateTSM}
\tilde x_c(t)=\left\{\begin{array}{ll}
x(t_k),&\mbox{if}\;\Gamma_x(x(t_k),e(t_k))\geq 0,\\
\tilde x_c(t_k),&\mbox{if}\;\Gamma_x(x(t_k),e(t_k))< 0,
\end{array}\right.
\end{align} 
where $e(t)=x(t)-\tilde x_c(t)$ 
and $\Gamma_x(x(t),e(t))$ is the triggering function that will be determined later.
The triggering times $t_0^x,t_1^x,t_2^x,\dots$ are given by $t_0^x=0$ and $t_{k+1}^x=\min_{i\in\Z_{\geq 0}}\{ih\mid ih> t_k^x,\Gamma_x(x(ih),e(ih))\geq 0\}.$
The control 
input to the plant, $u(t)$, is given as 
\begin{align}
u(t)=k(\tilde x_c(t)).\label{inputtriggerstate}
\end{align} 
\begin{figure}[!ht]
	\centering
	\begin{subfigure}{0.45\linewidth}
		\centering
		\includegraphics[height=3.55cm]{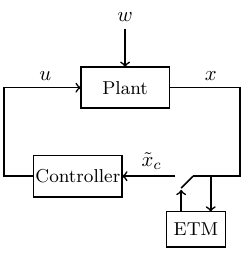}
		\caption{}
	\end{subfigure}\hfill
	\begin{subfigure}{0.55\linewidth}
		\centering
		\includegraphics[height=3.55cm]{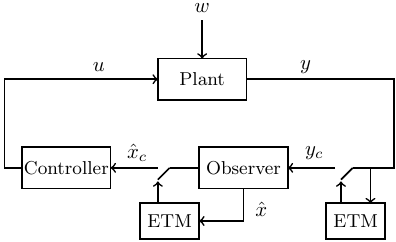}
		\caption{}
	\end{subfigure}
	\caption{(a) Configuration of the state feedback PETC (b) Configuration of the observer-based output feedback PETC }\label{figstate}
\end{figure}
Fig.\ref{figstate} (b)  shows the configuration of implementing the observer-based output feedback PETC, where  ETMs exist in both the communication and actuation
channels. The plant is given in \eqref{NLsys1} and the output is 
\begin{align}
y(t)&=g(x(t))\label{NLsys2}
\end{align}
where $y\in\R^{n_y}$ and $g$ is a continuous function. The observer is 
\begin{align}
\dot{\hat x}=\varphi(\hat x,u,y)\label{obser}
\end{align}
where $\hat x\in\R^{n_x}$, $\varphi$ is a continuously differentiable function, and the observer-based controller is given as $u(t)=k(\hat x(t))$ where $k(\cdot)$ is a continuous function. Assume that $\varphi(\cdot)$ and $k(\cdot)$ are designed for \eqref{NLsys1} and \eqref{NLsys2} such that without ETMs, the solution to  the closed-loop system exists for all time and all initial conditions, $\hat x$ asymptotically converges to $x$ when $w=0$, and the system \eqref{NLsys1} implementing the controller $u(t)$ is ISS w.r.t. $w$. When the ETMs are implemented, the output of the plant, $y(t)$, is sampled at each sampling time $t_k$. The input to the observer, $y_c(t)$, is updated only when the event-triggering condition for the output is satisfied. Specifically, $y_c(t)$ is a left-continuous, piecewise constant signal that is  defined for $t\in(t_k,t_{k+1}]$ as 
\begin{align}
y_c(t)=\left\{\begin{array}{ll}
y(t_k),&\mbox{if}\;\Gamma_y(y(t_k),y_e(t_k))\geq 0,\\
y_c(t_k),&\mbox{if}\;\Gamma_y(y(t_k),y_e(t_k))< 0,
\end{array}\right.\label{eqyc}
\end{align} 
where $y_e(t)=y_c(t)-y(t)$ 
and $\Gamma_y(y(t),y_e(t))$ is the triggering function of the output that will be determined later.
The triggering times $t_0^y,t_1^y,t_2^y,\dots$ are given by $t_0^y=0$ and $t_{k+1}^y=\min_{i\in\Z_{\geq 0}}\{ih\mid ih> t_k^y,\Gamma_y(y(ih),y_e(ih))\geq 0\}.$
Under ETMs, the observer \eqref{obser} becomes
\begin{align}
\dot{\hat x}=\varphi(\hat x,u,y_c).\label{obsersampled}
\end{align}
The input to the plant, $u(t)$, is updated only when the event-triggering condition for the input is satisfied. Specifically, define a left-continuous, piecewise constant signal $\hat x_c(t)$ for $t\in(t_k,t_{k+1}]$ as
\begin{align}
\hat x_c(t)=\left\{\begin{array}{ll}
\hat x(t_k),&\mbox{if}\;\Gamma_u(\hat x(t_k),x_e(t_k))\geq 0,\\
\hat x_c(t_k),&\mbox{if}\;\Gamma_u(\hat x(t_k),x_e(t_k))< 0,
\end{array}\right.\label{eqxc}
\end{align} 
where $x_e(t)=\hat x_c(t)-\hat x(t)$
and $\Gamma_u(\hat x(t),x_e(t))$ is the  triggering function of the
input that will be determined later.
The triggering times $t_0^u,t_1^u,t_2^u,\dots$ are given by $t_0^u=0$ and $t_{k+1}^u=\min_{i\in\Z_{\geq 0}}\{ih\mid ih> t_k^u,\Gamma_u(\hat x(ih),x_e(ih))\geq 0\}.$ 
Then the control input to the plant, $u(t)$,  is given as 
\begin{align}
u(t)=k(\hat x_c(t)).\label{outputtriggerstate}
\end{align}
Systems that are implemented with ETMs are impulsive systems, which 
evolve continuously based on ODEs most of the time and exhibit impulses at some instances. 
Clearly, for systems implemented with ETMs, the impulses happen when the triggering conditions are met. 
Inspired by \cite{hespanha2008lyapunov} and \cite{lin1995various}, the input-to-state stability of impulsive systems w.r.t. a given set is defined below.


\begin{definition}\label{dfnimpulISS}Consider the following impulsive system
	\begin{align}\label{dynimpul}
	\begin{cases}
	\dot{x}(t)=f(x(t),u(t)),t\in(T_i,T_{i+1}],\\
	x^+(t)=g(x(t),u(t)),t=T_i,
	\end{cases}
	\end{align}
	where $f$ is locally Lipschitz, $i\in\Z_{\geq 0}$, $\{T_0,T_1,T_2,\dots\}$ is a sequence of impulsive times with $T_0<T_1<\dots$, the state $x(t)\in\R^n$ is absolutely continuous between impulses, $u(t)\in\R^m$ is a locally bounded Lebesgue-measurable input, and $x^+(t):=\lim_{s\rightarrow t^+}x(s)$. 
	Given a time sequence $\{T_i\}$, the impulsive system \eqref{dynimpul} is ISS w.r.t. a given non-empty and closed set $\mathcal{A}$ if there exist functions $\beta\in\mathcal{KL}$ and $\gamma\in\mathcal{K}_\infty$, such that for every initial condition $x(T_0)$ and every admissible input $u$, the solution to \eqref{dynimpul} exists globally and satisfies
	\begin{align}\label{ineqISSimpul}
	\|x(t)\|_{\mathcal{A}}\leq \beta(\|x(T_0)\|_{\mathcal{A}},t-T_0)+\gamma(\|u\|_{[T_0,T]})
	\end{align}
	where $\|\cdot\|_I$ denotes the supremum norm on an interval $I$. 
	The impulsive system \eqref{dynimpul} is \emph{uniformly ISS} w.r.t. $\mathcal{A}$ over a given class $\mathcal{S}$ of admissible sequences of impulse times if there exist functions $\beta\in\mathcal{KL}$ and $\gamma\in\mathcal{K}_\infty$ that are independent of the choice of the time sequence, such that  \eqref{ineqISSimpul} holds for every time sequence in $\mathcal{S}$.
\end{definition}

In the following, the closed-loop system implemented with ETMs is called \emph{uniformly ISS}, or just ISS for short, w.r.t. a given (non-empty and closed) set $\mathcal{A}$, if it is uniformly ISS over all impulsive times generated by the periodic event-triggering mechanisms. It should be noted that the impulsive times generated by the periodic event-triggering mechanisms have no accumulation point (i.e. Zeno phenomenon is avoided) since the inter-execution times are lower bounded by the sampling period. 

The PETC design problems that will be investigated in this paper are the following: {\color{black}\emph{1). Given the configuration in  Fig.\ref{figstate} (a), design the sampling period $h$ and the triggering function $\Gamma_x(x,e)$ such  that the closed-loop system is ISS w.r.t. $w$;
	2). Given the configuration in  Fig.\ref{figstate} (b), design the sampling period $h$ and the triggering functions $\Gamma_y(y,y_e)$, $\Gamma_u(\hat x,x_e)$ such  that the closed-loop system is ISS w.r.t. $w$; 3) For incrementally quadratic systems, find LMI conditions to determine the sampling period and triggering functions for the configurations in Fig.\ref{figstate}.}}



\section{Input-to-state Stabilization Using PETC}\label{sec:general}

{\color{black}This section investigates input-to-state stabilization of  nonlinear systems affected by
	disturbances under PETC mechanisms. 
The overall system is modeled as an impulsive system where the continuous dynamics are assumed to be ISS while the discrete dynamics are not. In order to ensure the input-to-state stability of the closed-loop system, the sampling period $h$ is chosen such that it satisfies the average dwell-time condition (e.g., see \cite{hespanha2008lyapunov,dashkovskiy2013input}) and is upper bounded by the maximum allowable sampling period (e.g., see \cite{abdelrahim2017robust,carnevale2007lyapunov,nesic2009explicit}), and the triggering functions are designed using the corresponding Lyapunov functions as well as  other available parameters. }

The following lemma from \cite{carnevale2007lyapunov} will be used to determine the interval of the sampling period $h$. 
\begin{lemma}\cite{carnevale2007lyapunov}\label{lem2}
	Let $\phi:[0,\tilde{\mathcal{T}}]\rightarrow \R$ be the solution of the following ODE:
	\begin{align}
	\dot{\phi}=-2\mu\phi-\gamma(\phi^2+1)\label{odephi}
	\end{align}
	with $\phi(0)=\lambda^{-1},0<\lambda<1$,  $\mu>0,\gamma>0$, and 
	\begin{align}
	\tilde{\mathcal{T}}(\mu,\gamma,\lambda)&=
	\begin{cases}
	\frac{1}{\mu r}arctan\left(\frac{r(1-\lambda)}{2\frac{\lambda}{1+\lambda}(\frac{\gamma}{\mu}-1)+1+\lambda}\right),\quad\; \gamma>\mu,\\
	\frac{1}{\mu}\frac{1-\lambda}{1+\lambda},\hskip 40.2mm \gamma=\mu,\\
	\frac{1}{\mu r}arctanh\left(\frac{r(1-\lambda)}{2\frac{\lambda}{1+\lambda}(\frac{\gamma}{\mu}-1)+1+\lambda}\right),\quad \gamma<\mu,
	\end{cases}\label{eqtilT}\\
	r&=\sqrt{\left|\left(\frac{\gamma}{\mu}\right)^2-1\right|}.\label{eqr}
	\end{align}
	Then, $\phi(\tau)\in[\lambda,\lambda^{-1}]$ for all $\tau\in[0,\tilde{\mathcal{T}}]$, and $\phi(\tilde{\mathcal{T}})=\lambda$.
\end{lemma}

With  $r$ given in \eqref{eqr}, define $\mathcal{T}(\mu,\gamma)$ as
\begin{align}\label{MSAP}
\mathcal{T}(\mu,\gamma)=
\begin{cases}
\frac{1}{\mu r}arctan(r),\quad \gamma>\mu,\\
\frac{1}{\mu},\quad\quad\quad\quad\quad\;\; \gamma=\mu,\\
\frac{1}{\mu r}arctanh(r),\;\; \gamma<\mu.
\end{cases}
\end{align}

\begin{remark}\label{remarkT}
	Clearly, $\mathcal{T}(\mu,\gamma)$ and $\tilde{\mathcal{T}}(\mu,\gamma,\lambda)$ are both positive, and  $\mathcal{T}(\mu,\gamma)=\tilde{\mathcal{T}}(\mu,\gamma,0)$. Furthermore, for fixed $\mu,\gamma$, $\tilde{\mathcal{T}}(\mu,\gamma,\cdot)$ is a strictly decreasing function, and $\tilde{\mathcal{T}}(\mu,\gamma,\lambda)\rightarrow 0$ as $\lambda\rightarrow 1$. 
\end{remark}

\subsection{State Feedback PETC Design}\label{subsec:state}
Consider the configuration in Fig.\ref{figstate} (a) where the plant is \eqref{NLsys1} and the state feedback controller is \eqref{inputtriggerstate}. Define $\tau\in\R_{\geq 0}$ as the clock variable and ${\bf x}_s(t)=\begin{pmatrix}
x(t)\\e(t)\\\tau(t)
\end{pmatrix}$, ${\bf x}_s^+=
\begin{pmatrix}
x(t^+)\\e(t^+)\\\tau(t^+)
\end{pmatrix}$. 
The closed-loop system with the ETM in Fig.\ref{figstate} (a) is expressed as an impulsive model as follows:
\begin{align}
\dot{{\bf x}}_s\!&=\!F_s(x,e,w)\!:=\!\begin{pmatrix}
\tilde f_s(x,e,w)\\
\tilde f_s(x,e,w)\\
1
\end{pmatrix},\;\; t\in(t_k,t_{k+1}],\label{impulstate1}\\
{\bf x}_s^+
\!&=\!G_s(x,e)\!:=\!\begin{pmatrix}
x\\
g_s(x,e)\\
0
\end{pmatrix},\;\;\quad\quad t=t_k,\label{impulstate2}
\end{align}
where 
\begin{align}
\tilde f_s(x,e,w)&=f(x,k(x-e),w),\nonumber\\
g_s(x,e)&=
\begin{cases}
{\bf 0},\;\;\mbox{if}\;\Gamma_x(x,e)\geq 0,\\
e,\;\;\mbox{if}\; \Gamma_x(x,e)< 0.
\end{cases}\nonumber
\end{align}



\begin{theorem}\label{thmgeneral}
	Consider the configuration shown in Fig.\ref{figstate} (a) where the plant is \eqref{NLsys1} and the controller is \eqref{inputtriggerstate}. Suppose that there exist positive numbers $\mu,\gamma,\alpha,d$, and a differentiable, positive definite, radially unbounded function $V_1(x):\R^{n_x}\rightarrow \R_{\geq 0}$ such that $\forall {\bf x}_s\;a.e.$, $\forall w$, 
	\begin{align}
	\nabla V({\bf x}_s)F_s(x,e,w)\leq -\alpha V({\bf x}_s)+d\|w\|^2\label{eqthmflow}
	\end{align}
	where $V({\bf x}_s)=V_1(x)+V_2(e,\tau)$, $V_2(e,\tau)=\phi(\tau)e^\top e$, $\phi$ is the solution of ODE \eqref{odephi}.
	Choose positive numbers $\alpha_0,s,h,\lambda$ satisfying $\alpha_0<\alpha$, $\lambda<1$ and 
	\begin{gather}
	\frac{\ln(1+s)}{\alpha_0}<h<\mathcal{T}(\mu,\gamma),\label{eqthm1h}\\
	h=\tilde{\mathcal{T}}(\mu,\gamma,\lambda),\label{eqthm1lam}\\
	(1+s)\lambda^2<1,\label{eqthm1slam}
	\end{gather} 
	where $\tilde{\mathcal{T}}(\mu,\gamma,\lambda)$ and $\mathcal{T}(\mu,\gamma)$ are defined in  \eqref{eqtilT} and \eqref{MSAP}.
	Let the initial condition of $\phi$ be $\phi(0)=\lambda^{-1}$. If the triggering function is chosen as 
	\begin{align}
	\Gamma_x(x,e)=(\lambda^{-1}-(1+s)\lambda)\|e\|^2-sV_1(x),\label{eq1thm1tri}
	\end{align}
	then the closed-loop system \eqref{impulstate1}-\eqref{impulstate2} is ISS w.r.t. the set $\{(x,e,\tau)|(x,e)=({\bf 0},{\bf 0})\}$.
\end{theorem} 
\begin{proof} By Lemma \ref{lem2}, $\phi(\tau)\in[\lambda,\lambda^{-1}]$ for any $\tau\in[0,h]$, and $\phi(h)=\lambda$. Because $V_1$ and $V_2$ are both positive definite, the function $V$ is  positive definite w.r.t. $x$ and $e$ (i.e., $V({\bf x}_s)\geq 0$ for any $x,e\in\R^{n_x}$, and $V({\bf x}_s)=0$ when $x=e={\bf 0}$, $V({\bf x}_s)\neq 0$ otherwise). Furthermore, $V({\bf x}_s)$ is differentiable and radially unbounded for any $x,e\in\R^{n_x}$.
	
	During the continuous dynamics when $t\in(t_k,t_{k+1}]$, inequality \eqref{eqthmflow} implies
	\begin{align}
	V({\bf x}_s(t))\geq \frac{d}{\alpha-\alpha_0}\|w(t)\|^2
	\Rightarrow \dot{V}({\bf x}_s(t))\leq -\alpha_0V({\bf x}_s(t)),\;\forall t\in(t_k,t_{k+1}] \;a.e.\label{A4}
	\end{align}
	where $\dot V({\bf x}_s)$ is the derivative of $V$ along \eqref{impulstate1}.
	
	At the impulse time when $t=t_k$, there are two cases. Note that $(1+s)\lambda^2<1$ implies $\lambda^{-1}-(1+s)\lambda>0$. 
	(i) If $\Gamma_x(x,e)<0$, the triggering condition is not met. Since $\Gamma_x(x,e)<0$ implies  $\lambda^{-1}\|e\|^2<(1+s)\lambda \|e\|^2+sV_1(x)$, it holds that $V({\bf x}_s^+)= V_1(x)+\lambda^{-1} \|e\|^2<(1+s)V({\bf x}_s).$ 
	(ii) If $\Gamma_x(x,e)\geq 0$, the triggering condition is met. Then from \eqref{impulstate2} and since $e(t_k^+)={\bf 0}$, it holds that $V({\bf x}_s^+)= V_1(x)\leq V({\bf x}_s)$.
	In summary, at the impulse time when $t=t_k$, 
	\begin{align}
	V({\bf x}_s^+)\leq (1+s)V({\bf x}_s)=e^{\ln(1+s)}V({\bf x}_s).\label{ineqjump}
	\end{align} 
	Then a bound for $V({\bf x}_s(t))$ can be shown using \eqref{A4} and \eqref{ineqjump} as follows. Clearly, there exists a sequence of times $t_0:=\hat{t}_0\leq \check {t}_1< \hat{t}_1 < \check{t}_2< \hat{t}_2\dots$ such that 
	\begin{align}
	V({\bf x}_s(t))\geq \frac{d}{\alpha-\alpha_0}\|w\|^2_{[t_0,t]},&\;\forall t\in(\hat t_j,\check{t}_{j+1}],\label{A5}\\
	V({\bf x}_s(t))\leq \frac{d}{\alpha-\alpha_0}\|w\|^2_{[t_0,t]},&\;\forall t\in(\check t_i,\hat{t}_{i}],\label{A6}
	\end{align}
	where $j=0,1,2,\dots$ and $i=1,2,\dots$.
	Now consider the case when the first interval $(t_0,\check t_1]$ is non-empty, i.e., $\check t_1>t_0$. If $\check t_1<\infty$, then between any two consecutive impulses $t_{k-1},t_k\in(t_0,\check t_1]$, from \eqref{A4} and \eqref{A5}, it follows that $\dot{V}({\bf x}_s(t))\leq -\alpha_0V({\bf x}_s(t))$, $\forall t\in(t_{k-1},t_k]$ a.e., which implies that 
	$V({\bf x}_s(t_k))\leq e^{-\alpha_0h}V({\bf x}_s(t_{k-1})).$ 
	From \eqref{ineqjump}, it follows that 
	$V({\bf x}_s(t_k^+))\leq e^{\ln(1+s)}V({\bf x}_s(t_k)).$ 
	Therefore, for any $t\in(t_0,\check t_1]$, it holds that
	\begin{align}
	V({\bf x}_s(t))
	&\leq e^{\frac{\ln(1+s)-\alpha_0h}{h}(t-t_0)} V({\bf x}_s(t_0)).
	\label{A7}
	\end{align}
	If $\check t_1=\infty$, then it is easy to see that \eqref{A7} holds for any $t\in(t_0,\infty)$. Note that $\frac{\ln(1+s)-\alpha_0h}{h}<0$ by the choice of $h$ in \eqref{eqthm1h}.
	Next, consider the case when $t\geq \check t_1$. For any subinterval $(\check{t}_i,\hat t_i],i\geq 1$ where $\hat t_i<\infty$, inequality \eqref{A6} holds. If $\hat t_i$ is not an impulse time, then \eqref{A6} holds for $t=\hat t_i$. If  $\hat t_i$ is an impulse time, then \eqref{ineqjump} implies that 
	\begin{align}
	V({\bf x}_s(\hat t_i^+))\leq e^{\ln(1+s)}\frac{d}{\alpha-\alpha_0}\|w\|^2_{[t_0,\hat t_i]}.\label{A8}
	\end{align}
	In either case, inequality \eqref{A8} holds. For any subinterval $(\check{t}_i,\hat t_i]$, $i\geq 1$, where $\hat t_i=\infty$, it is easy to see that inequality \eqref{A8} also holds. In summary, \eqref{A8} holds for any subinterval $(\check{t}_i,\hat t_i],i\geq 1$.
	
	For any subinterval $(\hat t_j,\check{t}_{j+1}],j\geq 1$, using the same argument that  derives \eqref{A7}, the following inequality holds for any $t\in (\hat t_j,\check{t}_{j+1}]$:
	\begin{align}
	V({\bf x}_s(t))&\leq e^{\frac{\ln(1+s)-\alpha_0h}{h}(t-\hat t_i)} V({\bf x}_s(\hat t_i))\nonumber\\
	&\leq e^{\ln(1+s)}\frac{d}{\alpha-\alpha_0}\|w\|^2_{[t_0,\hat t_j]}.\label{A9}
	\end{align}
	Combing \eqref{A7}, \eqref{A8} and \eqref{A9}, it holds that 
	\begin{align*}
	V({\bf x}_s(t))\leq\max\{e^{\frac{\ln(1+s)-\alpha_0h}{h}(t-t_0)} V({\bf x}_s(t_0)),
	e^{\ln(1+s)}\frac{d}{\alpha-\alpha_0}\|w\|^2_{[t_0,t]}\},\;\forall t\geq t_0
	\end{align*} 
	Since $e^{\frac{\ln(1+s)-\alpha_0h}{h}(t-t_0)}$ is a strictly decreasing function for $t\geq t_0$, and $V$ is positive definite and radially unbounded for any $x,e\in\R^{n_x}$, by the standard argument for ISS (e.g., see \cite{sontag2008iss,lin1995various,hespanha2008lyapunov,dashkovskiy2013input}), it can be concluded that \eqref{ineqISSimpul} holds with the set $\mathcal{A}:=\{(x,e,\tau)|(x,e)=({\bf 0},{\bf 0})\}$. 
\end{proof}
{\color{black}The major assumption in Theorem \ref{thmgeneral} is that a sum-type ISS-Lyapunov function exists for the continuous dynamics \eqref{impulstate1}. Under this assumption, the sampling period and the triggering function can be always found, such that $x,e$ will converge globally to a neighborhood of the origin whose size depends on the norm of $w$. }
\begin{remark}\label{remarkpara} 
The sampling period $h$ is related to the system dynamics through $\mu,\gamma$, \eqref{eqthm1h} and \eqref{eqthm1lam}. 
Whenever $\mu,\gamma,\alpha$ satisfying  \eqref{eqthmflow} are found, $\alpha_0,s,h,\lambda$ satisfying \eqref{eqthm1h}-\eqref{eqthm1slam} always exist.
	Specifically, since $\ln(1+s)\rightarrow 0$ as $s\rightarrow 0^+$, there always exist $\alpha_0,s,h$ satisfying \eqref{eqthm1h}. Because  $\mathcal{T}(\mu,\gamma)$ and $\tilde{\mathcal{T}}(\mu,\gamma,0)$ have the properties stated in Remark \ref{remarkT}, there always exists $\lambda$ satisfying \eqref{eqthm1lam}. If \eqref{eqthm1slam} does not hold with such $s$ and $\lambda$, then it is always possible to find a smaller $s$ such that \eqref{eqthm1slam} holds, while still guaranteeing that \eqref{eqthm1h} holds. Therefore, $\alpha_0,s,h,\lambda$ can always be found. On the other hand, the choices of  $s,h,\alpha_0$ will affect the triggering frequencies of the ETM by changing $\Gamma_x(x,e)$. Moreover, different values of $s,h,\alpha_0$ will also affect the estimate impact of the disturbance $w$ through the term $e^{\ln(1+s)}\frac{d}{\alpha-\alpha_0}\|w\|^2_{[t_0,t]}$. 
	
\end{remark}



\subsection{Output Feedback PETC Design}\label{subsec:NL}

Consider the configuration in Fig.\ref{figstate} (b) where the plant is  \eqref{NLsys1}, the output is \eqref{NLsys2}, the observer is \eqref{obsersampled} and the observer-based output feedback controller is \eqref{outputtriggerstate}. 
Define 
$\hat e(t)=x(t)-\hat x(t)$, $\xi(t)=
\begin{pmatrix}
x(t)\\
\hat e(t)
\end{pmatrix}$, and 
$\eta(t)=\begin{pmatrix}
y_e(t)\\x_e(t)\\
\end{pmatrix}$. Define $\tau$ as a clock variable, and ${\bf x}_o(t)=\begin{pmatrix}
\xi(t)\\\eta(t)\\\tau(t)
\end{pmatrix}$, ${\bf x}_o^+=
\begin{pmatrix}
\xi(h^+)\\\eta(h^+)\\\tau(h^+)
\end{pmatrix}$. 
Then the closed-loop system with ETMs in Fig.\ref{figstate} (b) 
is expressed as an impulsive model as follows:
\begin{align}
\dot{{\bf x}}_o\!&=\!F_o(\xi,\eta,w):=\begin{pmatrix}
\tilde f_o^1(\xi,\eta,w)\\
\tilde f_o^2(\xi,\eta,w)\\
1
\end{pmatrix},\;t\in(t_k,t_{k+1}],\label{impuloutput1}\\
{\bf x}_o^+\!&=\!G_o(\xi,\eta):=\begin{pmatrix}
\xi\\
g_o(\xi,\eta)\\
0
\end{pmatrix},\;\;\quad\quad t=t_k,\label{impuloutput2}
\end{align}
where 
\begin{align}
\tilde f_o^1(\xi,\eta,w)&=\begin{pmatrix}
f(x,k(\hat x_c),w)\\
f(x,k(\hat x_c),w)-\varphi(\hat x,k(\hat x_c),w)
\end{pmatrix},\nonumber\\
\tilde f_o^2(\xi,\eta,w)&=\begin{pmatrix}
\nabla g(x)\cdot f(x,k(\hat x_c),w)\\
-f(x,k(\hat x_c),w)+\varphi(\hat x,k(\hat x_c),w)
\end{pmatrix},\nonumber\\
g_o(\xi,\eta)&=\begin{pmatrix}
g_o^1(\xi,\eta)\\
g_o^2(\xi,\eta)
\end{pmatrix},\nonumber\\
g_o^1(\xi,\eta)&=
\begin{cases}
{\bf 0},\;\;\!\mbox{if}\;\Gamma_y(y,y_e)\geq 0,\\
y_e,\mbox{if}\; \Gamma_y(y,y_e)< 0,
\end{cases}\nonumber\\
g_o^2(\xi,\eta)&=\begin{cases}
{\bf 0},\;\;\!\mbox{if}\; \Gamma_u(\hat x,x_e)\geq 0,\\
x_e,\mbox{if}\; \Gamma_u(\hat x,x_e)< 0.
\end{cases}\nonumber
\end{align}


\begin{theorem}\label{thmgeneraloutput}
	Consider the configuration shown in Fig.\ref{figstate} (b) where the plant, output, observer and controller are given by \eqref{NLsys1}, \eqref{NLsys2}, \eqref{obsersampled} and \eqref{outputtriggerstate}, respectively.
	Suppose that there exist positive numbers $\mu_1,\mu_2,\gamma_1,\gamma_2$, $c_1,c_2,\alpha,d$, and differentiable, positive definite, radially unbounded functions $V_1(\xi):\R^{2n_x}\rightarrow \R_{\geq 0}$, $V_3(y):\R^{n_y}\rightarrow \R_{\geq 0}$ and $V_4(\hat x):\R^{n_x}\rightarrow \R_{\geq 0}$ such that $\forall {\bf x}_o\;a.e.$, $\forall w$, 
	\begin{align}
	&\nabla V({\bf x}_o)F_o(\xi,\eta,w)\leq -\alpha V({\bf x}_o)+d\|w\|^2,\label{eqthmoutflow}\\
	&c_1V_3(g(x))+c_2V_4(\hat x)\leq V_1(\xi),\label{ineq5}
	\end{align}
	where $V({\bf x}_o)=V_1(\xi)+V_2(\eta,\tau)$, $V_2(\eta,\tau)=c_1\phi_1y_e^\top y_e+c_2\phi_2x_e^\top x_e$, 
	and  $\phi_i(i=1,2)$ is the solution of ODE $\dot{\phi}_i=-2\mu_i\phi_i-\gamma_i(\phi_i^2+1)$.  
	Choose positive numbers $\alpha_0,s,h,\lambda_1,\lambda_2$ satisfying $\alpha_0<\alpha$, $\lambda_1<1$, $\lambda_2<1$, and 
	\begin{align}
	&\frac{\ln(1+s)}{\alpha_0}<h<\min\{\mathcal{T}(\mu_1,\gamma_1),\mathcal{T}(\mu_2,\gamma_2)\},\label{eqthm2c1}\\
	&h=\tilde{\mathcal{T}}(\mu_1,\gamma_1,\lambda_1),\;\;h=\tilde{\mathcal{T}}(\mu_2,\gamma_2,\lambda_2),\label{eqthm2c2}\\
	&(1+s)\lambda_1^2<1,\;\;(1+s)\lambda_2^2<1,\label{eqthm2c3}
	\end{align} 
	where $\tilde{\mathcal{T}}(\mu,\gamma,\lambda)$ and $\mathcal{T}(\mu,\gamma)$ are defined in  \eqref{eqtilT} and \eqref{MSAP}.
	Let the initial condition of $\phi_i$ be $\phi_i(0)=\lambda_i^{-1}$ for $i=1,2$.
	If the triggering functions are chosen as 
	\begin{align}
	\Gamma_y(y,y_e)&=(\lambda_1^{-1}-(1+s)\lambda_1)\|y_e\|^2-sV_3(y),\label{thm2triy}\\
	\Gamma_u(\hat x, x_e)&=(\lambda_2^{-1}-(1+s)\lambda_2)\|x_e\|^2-sV_4(\hat x),\label{thm2trix}
	\end{align}
	then the closed-loop system \eqref{impuloutput1}-\eqref{impuloutput2} is ISS w.r.t. the set $\{(x,\hat e,\tau)|(x,\hat e)=({\bf 0},{\bf 0})\}$.
\end{theorem}

\begin{proof} 
	By Lemma \ref{lem2}, $\phi_i(\tau)\in[\lambda_i,\lambda_i^{-1}]$ for any $\tau\in[0,h]$, and $\phi_i(h)=\lambda_i$, $i=1,2$.
	Because $V_1,V_2$ are both positive definite, the function $V$ is  positive definite  w.r.t. $\xi$ and $\eta$. (i.e., $V({\bf x}_o)\geq 0$ for any $\xi\in\R^{2n_x},\eta\in\R^{n_x+n_y}$, and $V({\bf x}_o)=0$ when $\xi=\eta={\bf 0}$, $V({\bf x}_o)\neq 0$ otherwise). Furthermore, $V({\bf x}_o)$ is differentiable and radially unbounded for any $\xi,\eta$.
	
	During the continuous dynamics when $t\in(t_k,t_{k+1}]$, the inequality \eqref{eqthmoutflow} holds. Hence, 
	\begin{align}
	&V({\bf x}_o(t))\geq \frac{d}{\alpha-\alpha_0}\|w(t)\|^2\nonumber\\
	\Rightarrow \;&\dot{V}({\bf x}_o(t))\leq -\alpha_0V({\bf x}_o(t)),\;\forall t\in(t_k,t_{k+1}] \;a.e.\label{B1}
	\end{align}
	where $\dot V({\bf x}_o)$ is the derivative of $V$ along \eqref{impuloutput1}.
	
	At the impulse time when $t=t_k$, there are four cases regarding satisfaction of the input and output triggering conditions. Note that $(1+s)\lambda_i^2<1$ implies $\lambda_i^{-1}-(1+s)\lambda_i>0$, for $i=1,2$. 
	(i) If $\Gamma_y(y,y_c)<0$ and $\Gamma_u(\hat x,x_e)<0$, the output and input triggering conditions are not met. Since $\Gamma_y(y,y_c)<0$, $\lambda_1^{-1}\|y_e\|^2< (1+s)\lambda_1\|y_e\|^2+sV_3(y)$; since $\Gamma_u(\hat x,x_e)<0$, $\lambda_2^{-1}\|x_e\|^2<(1+s)\lambda_2\|x_e\|^2+sV_4(\hat x)$. Therefore, $V({\bf x}_o^+)
	=V_1(\xi)+c_1\lambda_1^{-1}\|y_e\|^2+c_2\lambda_2^{-1}\|x_e\|^2
	<V_1(\xi)+c_1(1+s)\lambda_1\|y_e\|^2+c_1sV_3(y)
	+c_2(1+s)\lambda_2\|x_e\|^2+c_2sV_4(\hat x)
	\leq(1+s)V({\bf x}_o).$
	(ii) If $\Gamma_y(y,y_e)<0$ and $\Gamma_u(\hat x,x_e)\geq 0$, then $V({\bf x}_o^+)= V_1(\xi)+
	c_1\lambda_1^{-1}\|y_e\|^2
	< V_1(\xi)+c_1(1+s)\lambda_1\|y_e\|^2+c_1sV_3(y)
	\leq (1+s)V({\bf x}_o).$
	(iii) If $\Gamma_y(y,y_e)\geq 0$ and $\Gamma_u(\hat x,x_e)< 0$, then $V({\bf x}_o^+)= V_1(\xi)+
	c_2\lambda_2^{-1}\|x_e\|^2
	< V_1(\xi)+c_2(1+s)\lambda_2\|x_e\|^2+c_2sV_4(\hat x)
	\leq (1+s)V({\bf x}_o).$
	(iv) If $\Gamma_y(y,y_e)\geq 0$ and $\Gamma_u(\hat x,x_e)\geq 0$, then $V({\bf x}_o^+)= V_1(\xi)\leq V({\bf x}_o).$
	In summary, at the impulse time when $t=t_k$, 
	\begin{align}
	V({\bf x}_o^+)\leq (1+s)V({\bf x}_o)=e^{\ln(1+s)}V({\bf x}_o).\label{ineqthm2jump2}
	\end{align} 
	From \eqref{B1} and \eqref{ineqthm2jump2}, the same argument as in the proof of Theorem \ref{thmgeneral} can be used to show that 
	$$
	V({\bf x}_o(t))\leq\max\{e^{\frac{\ln(1+s)-\alpha_0h}{h}(t-t_0)} V({\bf x}_o(t_0)),
	e^{\ln(1+s)}\frac{d}{\alpha-\alpha_0}\|w\|^2_{[t_0,t]}\},\forall t\geq t_0.
	$$
	Since $e^{\frac{\ln(1+s)-\alpha_0h}{h}(t-t_0)}$ is a strictly decreasing function for $t\geq t_0$, and $V$ is positive definite and radially unbounded for any $\xi,\eta$,  by the standard argument  for ISS, it can be concluded that the closed-loop system  \eqref{impuloutput1}-\eqref{impuloutput2} is ISS w.r.t. the set $\{(\xi,\eta,\tau)|(\xi,\eta)=({\bf 0},{\bf 0})\}$, and therefore, it is ISS w.r.t. the set $\{(x,\hat e,\tau)|(x,\hat e)=({\bf 0},{\bf 0})\}$.  
\end{proof}

\begin{remark}
	Discussion similar to that in Remark \ref{remarkpara} also holds for the output feedback case. In particular, when parameters $\mu_1,\mu_2,\gamma_1,\gamma_2,c_1,c_2,\alpha,d$ are found, numbers $\alpha_0,s,h,\lambda_1,\lambda_2$ satisfying \eqref{eqthm2c1}-\eqref{eqthm2c3} always exist.	
\end{remark}


{\color{black}
\begin{remark}\label{remarkotheremu}
PETC design for nonlinear systems with exogenous disturbances was investigated using the emulation-based approach under a hybrid system framework in \cite{wang2016stabilization,wang2018periodic,wang2019periodic},  where the basic idea is to construct a hybrid Lyapunov function for the overall system by assuming that  $x$ and $e$ subsystems are both ISS and using small-gain techniques, which are known to be conservative in general. In contrast, Theorem \ref{thmgeneral} and \ref{thmgeneraloutput} above provide an impulsive system approach to solve PETC design of nonlinear systems, where the main assumption is that the continuous dynamics are ISS (i.e., \eqref{eqthmflow} in Theorem \ref{thmgeneral} and \eqref{eqthmoutflow}-\eqref{ineq5} in Theorem \ref{thmgeneraloutput}) and the key idea is to determine the valid interval of the sampling period (i.e., \eqref{eqthm1h}-\eqref{eqthm1slam} in Theorem \ref{thmgeneral} and \eqref{eqthm2c1}-\eqref{eqthm2c3} in Theorem \ref{thmgeneraloutput}) using  techniques from \cite{hespanha2008lyapunov,dashkovskiy2013input} and  \cite{carnevale2007lyapunov,nesic2009explicit}. Since the Lyapunov function of the overall system is chosen as that of the continuous dynamics, sufficient conditions of Theorem \ref{thmgeneral} and \ref{thmgeneraloutput} are imposed on the continuous dynamics directly. Although it is difficult to compare quantitatively the conservatism of the sufficient conditions  proposed above and those in \cite{wang2016stabilization,wang2018periodic,wang2019periodic}, Theorem \ref{thmgeneral} and  \ref{thmgeneraloutput} provide a novel and promising approach that is significantly different from existing results to tackle PETC design problems. 
\end{remark}
}

\section{PETC for Incrementally Quadratic Systems}\label{sec:dQC}

{\color{black}In this section, sufficient conditions in Theorem \ref{thmgeneral} and \ref{thmgeneraloutput} are expressed as LMI conditions, which can be solved by convex program solvers, for incrementally quadratic nonlinear systems. Therefore,  the sampling period and the triggering functions can be computed systematically  for a large class of nonlinear systems.}
Suppose that the plant in Fig.\ref{figstate} (a) and Fig.\ref{figstate} (b)  is an incrementally quadratic nonlinear system given as
\begin{align}\label{dyn1}
\begin{cases}
\dot x=Ax+Bu+Ep(q)+E_ww,\\
q=C_qx,
\end{cases}
\end{align}
where $x\in\R^{n_x}$ is the state, $u\in\R^{n_u}$ is the control input, $p:\R^{n_q}\rightarrow \R^{n_p}$ is a function representing the known nonlinearity, $w\in\R^{n_w}$ is the unknown external disturbance, and $A\in\R^{n_x\times n_x},B\in\R^{n_x\times n_u},C_q\in\R^{n_q\times n_x},E\in\R^{n_x\times n_p},E_w\in\R^{n_x\times n_w}$ are constant matrices with proper sizes.
The characterization of  $p$ is based on the incremental multiplier matrix defined below.
\begin{definition}\label{def:delQC}\cite{accikmecse2011observers}
	Given a function $p:\R^{n_q}\rightarrow \R^{n_p}$, a symmetric matrix $M\in\R^{(n_q+n_p)\times (n_q+n_p)}$ is called an \emph{incremental multiplier matrix} 
	for $p$ if it satisfies the following \emph{incremental quadratic constraint}: 
	\begin{equation}\label{eq:delQC}
	\begin{pmatrix}
	\delta q\\
	\delta p
	\end{pmatrix}^\top M 
	\begin{pmatrix}
	\delta q\\
	\delta p
	\end{pmatrix}\geq 0,\;\;\forall q_1, q_2\in\R^{n_q},
	\end{equation} 
	where $\delta q=q_2-q_1$, $\delta p=p(q_2)-p(q_1)$.	
\end{definition}
The incrementally quadratic nonlinear systems subsume globally Lipschitz nonlinear systems and many other common nonlinear systems \cite{accikmecse2011observers,xuobserverarxiv18}. Given a nonlinearity $p$, its incremental multiplier matrix that satisfies \eqref{eq:delQC} is not unique. 
Assume that $p({\bf 0}_{n_q})={\bf 0}_{n_p}$ in the following, which implies that $\begin{pmatrix}
q\\
p
\end{pmatrix}^\top M 
\begin{pmatrix}
q\\
p
\end{pmatrix}\geq 0,\;\forall q\in\R^q.$

\subsection{State Feedback PETC Design  For Incrementally Quadratic Nonlinear Systems}

Consider the configuration in Fig.\ref{figstate} (a) where the full-state information is available. Suppose that the plant is given as \eqref{dyn1}-\eqref{eq:delQC}
and the controller is $u=K_1x+K_2p(C_qx)$ 
where $K_1\in\R^{n_u\times n_x},K_2\in\R^{n_u\times n_p}$. 
In the following, matrices $K_1,K_2$ are assumed to be chosen such that the closed-loop system in Fig.\ref{figstate} (a) without ETM is ISS (e.g., by using the results of \cite{xuobserverarxiv18}). 
With ETMs in Fig.\ref{figstate} (a), the control input to the plant is given as 
\begin{align}
u(t)=K_1\tilde x_c(t)+K_2p(C_q\tilde x_c(t))\label{inputupdatestate}
\end{align}
where $\tilde x_c$ is defined in \eqref{stateTSM}. The closed-loop system in Fig.\ref{figstate} (a)
is expressed in the form of \eqref{impulstate1}-\eqref{impulstate2} with
$\tilde f_s(x,e,w)=(A+BK_1)x-BK_1e+(E+BK_2)p+BK_2\delta  \tilde p+E_ww$,
where $\delta \tilde p=p(q+\delta \tilde q)-p(q)$, $\delta \tilde q=-C_qe.$


\begin{theorem}\label{thmstate}
	Consider the configuration in Fig.\ref{figstate} (a) where the plant is \eqref{dyn1}-\eqref{eq:delQC} and the control input is  \eqref{inputupdatestate}. Given $\alpha>0$, suppose that there exist positive numbers $\mu,\gamma,d$, non-negative numbers $\sigma_1,\sigma_2$, matrix $P\in\R^{n_x\times n_x}$ where $P=P^\top\succ 0$, such that \eqref{LMIstate} holds where $\Psi,S_1,S_2$ are given as
	\begin{align}
	&\begin{cases}
	\Psi=P(A+BK_1)+(A+BK_1)^\top P+\alpha P,\\
	S_1=\begin{pmatrix}
	C_q, {\bf 0}_{n_q\times (2n_x+2n_p+n_w) }\\
	{\bf 0}_{n_p\times 2n_x}, I_{n_p},{\bf 0}_{n_p\times (n_x+n_p+n_w)}
	\end{pmatrix},\\
	S_2=\begin{pmatrix}
	{\bf 0}_{n_q\times n_x},-C_q, {\bf 0}_{n_q\times (n_x+2n_p+n_w) }\\
	{\bf 0}_{n_p\times (2n_x+n_p)}, I_{n_p},{\bf 0}_{n_p\times (n_x+n_w)}
	\end{pmatrix}.
	\end{cases}\label{eqthm1}
	\end{align}
	Choose positive numbers $\alpha_0,s,h,\lambda$ satisfying $\alpha_0<\alpha$, $\lambda<1$ and $\frac{log(1+s)}{\alpha_0}<h<\mathcal{T}(\mu,\gamma)$,
	$h=\tilde{\mathcal{T}}(\mu,\gamma,\lambda)$,
	$(1+s)\lambda^2<1$.
	If the triggering function is chosen as 
	$
	\Gamma_x(x,e)=(\lambda^{-1}-(1+s)\lambda)\|e\|^2-sx^\top Px,
	$
	then the closed-loop system in Fig.\ref{figstate} (a)  is ISS w.r.t. the set $\{(x,e,\tau)|(x,e)=({\bf 0},{\bf 0})\}$.
\end{theorem}
\begin{figure*}[!ht]
	\begin{align}
	&\begin{pmatrix}
	\Psi & -PBK_1 & P(E\!+\!BK_2) & PBK_2 & (A\!+\!BK_1)^\top &PE_w\\
	* & -\gamma I & {\bf 0} & {\bf 0} & -(BK_1)^\top\!+\!(\frac{\alpha}{2}\!-\!\mu)I & {\bf 0} \\
	* & * & {\bf 0} & {\bf 0} & (E\!+\!BK_2)^\top & {\bf 0} \\
	* & * & * & {\bf 0} & (BK_2)^\top & {\bf 0} \\
	* & * & * & * & -\gamma I & E_w \\
	* & * & * & * & * &-dI \\
	\end{pmatrix}\nonumber\\
	&\quad\quad\quad +\sigma_1 S_1^\top MS_1+\sigma_2 S_2^\top MS_2\preceq 0\label{LMIstate}
	\end{align}
\end{figure*}
Theorem \ref{thmstate} shows that $h$ and $\Gamma_x(x,e)$ can be constructed for the state feedback case by solving LMI \eqref{LMIstate}. The proof of Theorem \ref{thmstate} is similar to that of Theorem \ref{thmoutput} in the following, and is omitted due to space limitation.

\subsection{Observer-based Output Feedback PETC Design For Incrementally Quadratic Nonlinear Systems}

Consider the configuration in Fig.\ref{figstate} (b) where the measured output
information is available. The plant is  \eqref{dyn1}-\eqref{eq:delQC} and the output is $y=Cx$ where  $y\in\R^{n_y}$ and $C\in\R^{n_y\times n_x}$. 
Suppose that the observer is
\begin{align}\label{obserCT}
\begin{cases}
\dot{\hat{x}}=A\hat x\!+\!Bu\!+\!E p(\hat q\!+\!L_1(\hat y\!-\!y))\!+\!L_2(\hat y\!-\!y),\\
\hat y=C\hat x,\\
\hat q=C_q\hat x,
\end{cases}
\end{align}
with $L_1\in\R^{n_q\times n_y},L_2\in \R^{n_x\times n_y}$, and the controller is 
\begin{align}\label{inputCT}
u(t)=K_1\hat x(t)+K_2 p(C_q\hat x(t)).
\end{align}  
In the following, matrices $L_1,L_2,K_1,K_2$ are assumed to be chosen such that the closed-loop system in Fig.\ref{figstate} (b) without ETMs is ISS (e.g., by using the results of \cite{xuobserverarxiv18}). 
With ETMs in Fig.\ref{figstate} (b), the observer becomes
\begin{align}\label{trobser}
\begin{cases}
\dot{\hat{x}}\!=\!A\hat x\!+\!Bu\!+\!E p(\hat q\!+\!L_1(\hat y\!-\!y_c))\!+\!L_2(\hat y\!-\!y_c),\\
\hat y\!=\!C\hat x,\\
\hat q\!=\!C_q\hat x.
\end{cases}
\end{align}
The observer-based controller now becomes
\begin{align}\label{inputupdate1}
u(t)=K_1\hat x_c(t)+K_2 p(C_q\hat x_c(t)).
\end{align} 
Then the closed-loop system in Fig.\ref{figstate} (b) is expressed in the form of  \eqref{impuloutput1}-\eqref{impuloutput2} with 
\begin{align*}
\tilde f_o^1(\xi,\eta,w)&=A_1\xi+A_2\eta+H_1p+H_2\delta  \check p+H_3\delta\hat p+H_4w,\\
\tilde f_o^2(\xi,\eta,w)&=A_3\xi+A_4\eta+H_5p+H_6\delta  \check p+H_7\delta\hat p+H_8w,
\end{align*}
where $\delta \hat p=p(q+\delta \hat q)-p(q)$, $\delta \hat q=C_q(x_e-\hat e)$, $\delta \check p=p(q+\delta \check q)-p(q)$, $\delta \check q=-(C_q+L_1C)\hat e-L_1y_e$, and
\begin{align*}
&A_1\!=\!
\begin{pmatrix}
A\!+\!BK_1&-BK_1\\
{\bf 0}& A\!+\!L_2 C
\end{pmatrix}\!,A_4\!=\!
\begin{pmatrix}
{\bf 0}& -CBK_1\\
L_2& -BK_1
\end{pmatrix}\!,\\
&A_2\!=\!
\begin{pmatrix}
{\bf 0}& BK_1\\
L_2& {\bf 0}
\end{pmatrix}\!,A_3\!=\!
\begin{pmatrix}
-C(\!A\!+\!BK_1\!)&CBK_1\\
-(\!A\!+\!BK_1\!)& A\!+\!BK_1\!+\!L_2 C
\end{pmatrix}\!,\\ 
&H_1\!=\!
\begin{pmatrix}
E\!+\!BK_2\\
{\bf 0}
\end{pmatrix}\!,H_2\!=\!
\begin{pmatrix}
{\bf 0}\\
-E
\end{pmatrix}\!,H_3\!=\!
\begin{pmatrix}
BK_2\\
{\bf 0}
\end{pmatrix}\!,\\
&H_4\!=\!
\begin{pmatrix}
E_w\\
E_w
\end{pmatrix}\!,H_5\!=\!
\begin{pmatrix}
-C(\!E\!+\!BK_2\!)\\
-(\!E\!+\!BK_2\!)
\end{pmatrix}\!,H_6\!=\!
\begin{pmatrix}
{\bf 0}\\
-E
\end{pmatrix}\!,\\
&H_7\!=\!
\begin{pmatrix}
-CBK_2\\
-BK_2
\end{pmatrix}\!,H_8\!=\!
\begin{pmatrix}
-CE_w\\
{\bf 0}
\end{pmatrix}\!.
\end{align*}		

\begin{figure*}[!ht]
	\begin{align}\label{LMI1}
	&\begin{pmatrix}
	PA_1\!+\!A_1^\top P\!+\!\alpha P & PA_2 & PH_1 & PH_2 & PH_3 & A_3^\top& PH_4\\
	* & R_1& {\bf 0} & {\bf 0} & {\bf 0} & A_4^\top\!+\!R_3^\top\!+\!\frac{\alpha}{2}I & {\bf 0}\\
	* & * & {\bf 0} & {\bf 0} &{\bf 0} & H_5^\top & {\bf 0} \\
	* & * & * & {\bf 0} & {\bf 0} & H_6^\top & {\bf 0} \\
	* & * & * & * & {\bf 0} & H_7^\top & {\bf 0} \\
	* & * & * & * & * & R_2 & H_8 \\
	* & * & * & * & * & * & -dI
	\end{pmatrix}\nonumber\\
	&\quad \quad +\sigma_1 S_1^\top MS_1+\sigma_2 S_2^\top MS_2+\sigma_3 S_3^\top MS_3\preceq 0
	\end{align}
	\begin{align}
	&\begin{cases}
	R_1=\begin{pmatrix}
	-a_1I_{n_y} & {\bf 0} \\
	{\bf 0} & -a_2I_{n_x}
	\end{pmatrix},\;
	R_2=\begin{pmatrix}
	-b_1I_{n_y} & {\bf 0} \\
	{\bf 0} & -b_2I_{n_x}
	\end{pmatrix},\;
	R_3=\begin{pmatrix}
	-\mu_1I_{n_y} & {\bf 0} \\
	{\bf 0} & -\mu_2I_{n_x}
	\end{pmatrix},\\
	S_1=\begin{pmatrix}
	C_q,{\bf 0}_{n_q\times (3n_x+2n_y+3n_p+n_w)}\\
	{\bf 0}_{n_p\times (3n_x+n_y)},I_{n_p},{\bf 0}_{n_p\times (n_x+n_y+2n_p+n_w)}
	\end{pmatrix},\\
	S_2=\begin{pmatrix}
	{\bf 0}_{n_q\times n_x},-(C_q+L_1C),-L_1, {\bf 0}_{n_q\times (2n_x+n_y+3n_p+n_w) },\\
	{\bf 0}_{n_p\times (3n_x+n_y+n_p)}, I_{n_p},{\bf 0}_{n_p\times (n_x+n_y+n_p+n_w)}
	\end{pmatrix},\\
	S_3=\begin{pmatrix}
	{\bf 0}_{n_q\times n_x},-C_q, {\bf 0}_{n_q\times n_y},C_q,{\bf 0}_{n_q\times (n_x+n_y+3n_p+n_w) }\\
	{\bf 0}_{n_p\times (3n_x+n_y+2n_p)}, I_{n_p},{\bf 0}_{n_p\times (n_x+n_y+n_w)}
	\end{pmatrix}.
	\end{cases}\label{formulathm2}
	\end{align}
\end{figure*}

\begin{theorem}\label{thmoutput}
	Consider  the configuration in Fig.\ref{figstate} (b) where the plant is \eqref{dyn1}-\eqref{eq:delQC}, the output is $y=Cx$, the observer is \eqref{trobser}, and the control input is \eqref{inputupdate1}.  Given $\alpha>0$, suppose that there exist positive numbers $\mu_1,\mu_2,a_1,a_2,b_1,b_2,d,\sigma_1,\sigma_2,\sigma_3$, and matrix $P\in\R^{2n_x\times 2n_x}$, $P=P^\top\succ 0$, such that \eqref{LMI1} holds
	where $R_1,R_2,R_3,S_1,S_2,S_3$ are given in \eqref{formulathm2}. Suppose that there exist 
	matrices  $P_1\in\R^{n_y\times n_y},P_1=P_1^\top\succ 0$, $P_2\in\R^{n_x\times n_x},P_2=P_2^\top\succ 0$, such that 
	\begin{align}
	\begin{pmatrix}
	c_1C^\top P_1 C&{\bf 0}\\
	{\bf 0}&c_2P_2
	\end{pmatrix}
	\!\preceq\! \begin{pmatrix}
	I_{n_x}&I_{n_x}\\
	{\bf 0}&-I_{n_x}
	\end{pmatrix}\!P\!\begin{pmatrix}
	I_{n_x}&{\bf 0}\\
	I_{n_x}&-I_{n_x}
	\end{pmatrix}\label{LMI2}
	\end{align}
	where $c_1=\sqrt{a_1/b_1}$, $c_2=\sqrt{a_2/b_2}$. 
	Choose positive numbers $\alpha_0,s,h,\lambda_1,\lambda_2$ satisfying $\alpha_0<\alpha$, $\lambda_1<1$, $\lambda_2<1$, and $\frac{log(1+s)}{\alpha_0}<h<\min\{\mathcal{T}(\mu_1,\gamma_1),\mathcal{T}(\mu_2,\gamma_2)\}$,
	$h=\tilde{\mathcal{T}}(\mu_1,\gamma_1,\lambda_1)$,$h=\tilde{\mathcal{T}}(\mu_2,\gamma_2,\lambda_2)$,
	$(1+s)\lambda_1^2<1$,$(1+s)\lambda_2^2<1,$
	where  $\gamma_1=\sqrt{a_1b_1}$, $\gamma_2=\sqrt{a_2b_2}$. 
	If the triggering functions are chosen as $\Gamma_y(y,y_e)=(\lambda_1^{-1}-(1+s)\lambda_1)\|y_e\|^2-sy^\top P_1 y$,
	$\Gamma_u(\hat x,x_e)=(\lambda_2^{-1}-(1+s)\lambda_2)\|x_e\|^2-s\hat x^\top P_2 \hat x$,
	then the closed-loop system in Fig.\ref{figstate} (b) is ISS w.r.t. the set $\{(x,\hat e,\tau)|(x,\hat e)=({\bf 0},{\bf 0})\}$.
\end{theorem}
\begin{proof} Define $V({\bf x}_o)=V_1(\xi)+V_2(\eta,\tau)$ where $V_1(x)=\xi^\top P\xi$, $V_2(\eta,\tau)=c_1\phi_1y_e^\top y_e+c_2\phi_2x_e^\top x_e$, ${\bf x}_o$ is defined in Subsec. \ref{subsec:NL}, and $\phi_i$ is the solution of ODE $\dot{\phi}_i=-2\mu_i\phi_i-\gamma_i(\phi_i^2+1)$ with the initial condition $\phi_i(0)=\lambda_i^{-1}$, for $i=1,2$. Define $V_3(y)=y^\top P_1y$ and $V_4(\hat x)=\hat x^\top P_2\hat x$. It is easy to see that if \eqref{eqthmoutflow} and \eqref{ineq5}  hold during the flow (i.e., when $t\in(t_k,t_{k+1}]$), then all the conditions of Theorem \ref{thmgeneraloutput} hold with $\Gamma_u,\Gamma_y$ given in \eqref{thm2triy}-\eqref{thm2trix}, and the conclusion follows immediately. 
	Define 
	$
	\varrho=\begin{pmatrix}
	\varrho_y\\
	\varrho_x
	\end{pmatrix}:=\begin{pmatrix}
	c_1\phi_1y_e\\
	c_2\phi_2x_e
	\end{pmatrix}
	$ 
	and 
	$\zeta=(\xi^\top,\eta^\top,p^\top,\delta\check p^\top,\delta\hat p,\varrho^\top,w^\top)^\top.$
	Clearly, $\varrho=Q\eta$, which implies that $V_2(\eta,\tau)=\eta^\top \varrho$. 
	During the flow \eqref{impuloutput1}, $\langle \nabla V({\bf x}_o), F_o(\xi,\eta,w) \rangle=\frac{\partial V_1}{\partial \xi}\tilde f_o^1(\xi,\eta,w)+\frac{\partial V_2}{\partial \eta}\tilde f_o^2(\xi,\eta,w) +\eta^\top \frac{\partial Q}{\partial \tau}\eta=2\xi^\top P(A_1\xi+A_2\eta+H_1p+H_2\delta  \check p+H_3\delta\hat p+H_4w)
	+2\varrho^\top(A_3\xi+A_4\eta+H_5p+H_6\delta  \check p+H_7\delta\hat p+H_8w)
	+\eta^\top R_1\eta+\varrho^\top R_2\varrho+2\eta^\top R_3\varrho.$ 
	Noting that 
	$\begin{pmatrix}
	q\\
	p
	\end{pmatrix}=S_1\zeta, \;
	\begin{pmatrix}
	\delta \check q\\
	\delta \check p
	\end{pmatrix}=S_2\zeta,\;\begin{pmatrix}
	\delta \hat q\\
	\delta \hat p
	\end{pmatrix}=S_3\zeta
	$, 
	it hold that $\sigma_1\zeta^\top S_1^\top MS_1\zeta\geq 0$, $\sigma_2\zeta^\top S_2^\top MS_2\zeta\geq 0$, $\sigma_3\zeta^\top S_3^\top MS_3\zeta\geq 0$. Multiplying the left-hand side and the right-hand side of \eqref{LMI1} by $\zeta^\top$ and $\zeta$, respectively, it follows that $2\xi^\top P(A_1\xi+A_2\eta+H_1p+H_2\delta  \check p+H_3\delta\hat p+H_4w)+2\varrho^\top(A_3\xi+A_4\eta+H_5p+H_6\delta  \check p+H_7\delta\hat p+H_8w)+\eta^\top R_1\eta+\varrho^\top R_2\varrho+2\eta^\top R_3\varrho+\alpha\xi^\top P\xi+\alpha\eta^\top \varrho-d\|w\|^2+\sigma_1\zeta^\top S_1^\top MS_1\zeta+ \sigma_2\zeta^\top S_2^\top MS_2\zeta+\sigma_3\zeta^\top S_3^\top MS_3\zeta\leq 0$. 
	Therefore, it is easy to obtain that 
	$\langle \nabla V({\bf x}_o), F_o(\xi,\eta,w) \rangle\leq-\alpha(\xi^\top P\xi+\eta^\top \varrho)+d\|w\|^2=-\alpha V({\bf x}_o)+d\|w\|^2$. 
	Therefore, \eqref{eqthmoutflow} holds during the flow.
	Since $\xi=\begin{pmatrix}
	I_{n_x}&{\bf 0}\\
	I_{n_x}&-I_{n_x}
	\end{pmatrix}\begin{pmatrix}
	x\\\hat x
	\end{pmatrix}$, multiplying $\begin{pmatrix}
	x\\\hat x
	\end{pmatrix}^\top$ and its transpose to the left-hand side and the right-hand side of \eqref{LMI2}, respectively, it follows that $c_1x^\top C^\top P_1 Cx+c_2\hat x^\top P_2 \hat x\leq \xi^\top P\xi$, which is equivalent to $c_1y^\top P_1 y+c_2\hat x^\top P_2\hat x\leq \xi^\top P\xi$. Therefore, \eqref{LMI2} implies that \eqref{ineq5} holds with $V_3=y^\top P_1 y,V_4=\hat x^\top P_2\hat x$. This completes the proof.
\end{proof}
{\color{black}
\begin{remark}
In \cite{wang2019periodic,luc2017periodic}, LMI-based sufficient conditions were given for
Lipschitz systems, which is a subset of incrementally quadratic systems considered above \cite{accikmecse2011observers}.	
\end{remark}}

\subsection{Special Case: Linear Control Systems}
PETC design for continuous-time linear  systems was investigated in  \cite{heemels2013periodic}. By letting $E={\bf 0}$, dynamics of \eqref{dyn1} becomes a linear system $\dot x=Ax+Bu+E_ww$, for which results in preceding subsections can be applied directly. 

For the configuration in Fig.\ref{figstate} (a), suppose that the state feedback controller implemented with ETM is $u(t)=K\tilde x_c(t)$
where $K\in\R^{n_u\times n_x}$ and $\tilde x_c$ is defined in \eqref{stateTSM}. Then the conditions of Theorem \ref{thmstate} becomes finding positive numbers $\mu,\gamma,d$, and a matrix $P=P^\top\succ 0$ such that
\begin{align*}
&\begin{pmatrix}
\Psi & -PBK & (A+BK)^\top &PE_w\\
* & -\gamma I & -(BK)^\top+(\frac{\alpha}{2}-\mu)I & {\bf 0} \\
* & * &  -\gamma I & E_w \\
* & * &  * &-dI \\
\end{pmatrix}\preceq 0
\end{align*}
where $\Psi=P(A+BK)+(A+BK)^\top P+\alpha P$. 

For the configuration in Fig.\ref{figstate} (b), suppose that the output is $y=Cx$ with $C\in\R^{n_y\times n_x}$,  the observer is $\dot{\hat{x}}=A\hat x+Bu+L(C\hat x-y_c)$ 
where $L\in\R^{n_x\times n_y}$, 
and the controller is $u(t)=K\hat x_c(t)$ 
where $\hat x_c$ is defined in \eqref{eqxc}. 
Then the conditions in Theorem \eqref{thmoutput} becomes finding positive numbers $\mu_1,\mu_2,a_1,a_2,b_1,b_2,d$, non-negative numbers $\sigma_1,\sigma_2,\sigma_3$, and a matrix $P=P^\top\succ 0$ such that 
\begin{align*}
&\begin{pmatrix}
\Psi & PA_2  & A_3^\top& PH_4\\
* & R_1& A_4^\top\!+\!R_3^\top\!+\!\frac{\alpha}{2}I & {\bf 0}\\
* & * &  R_2 & H_8 \\
* & * &  * & -dI
\end{pmatrix}\preceq 0
\end{align*}
where $\Psi=PA_1\!+\!A_1^\top P\!+\!\alpha P$, $A_i(i=1,2,3,4)$ and $H_4,H_8$ are given in the preceding subsection, $R_i(i=1,2,3)$ are given in \eqref{formulathm2}.

\section{Simulation Examples}\label{sec:exam}

\begin{example}\label{ex1}
	Consider the following plant given in \cite{borgers2018periodic,nesic2009explicit}:
	\begin{align*}
	\dot{x}=x^2-x^3+u+0.1w
	\end{align*}
	where a state feedback controller is given as 
	$u(t)=-2x(t)$. 
	For the configuration shown in Fig.\ref{figstate} (a), the controller becomes $u(t)=-2\tilde x_c(t)$ as in \eqref{inputtriggerstate}. The closed-loop system can be expressed as an impulsive model  \eqref{impulstate1}-\eqref{impulstate2} with $\tilde f(x,e,w)=x^2-x^3-2x+2e+w$. By using the SOSTOOLS toolbox (see \cite{papachristodoulou2013sostools}), it can be verified that \eqref{eqthmflow} holds with $V_1(x)=1.0192x^2-0.1298x^3+0.4784x^4$,  $\mu=0.4941,\gamma=4.4302,\alpha=1.2,d=0.1$. Since $\mathcal{T}(\mu,\gamma)=0.3314$, pick $s=0.1,\alpha_0=1.1$, and $h=0.1$, such that \eqref{eqthm1h} holds. Then there exists $\lambda=0.6$ such that $h=\tilde{\mathcal{T}}(\mu,\gamma,\lambda)$, and one can verify that $(1+s)\lambda^2<1$. 
	By Theorem \ref{thmgeneral}, the triggering condition is chosen as $\Gamma_x(e,x)=1.0067e^2-0.1V_1(x).$
	The simulation results for two sets of initial states and disturbance bounds are shown in Fig \ref{figex1}, where trajectories of the state $x$ and the input $u$ are depicted. The red lines (resp. blue  lines) indicate the simulation with the initial state $x(0)=0.3$ (resp. $x(0)=-0.4$) where the disturbance $w$ satisfying $\|w\|_\infty\leq 0.8$ (resp. $\|w\|_\infty\leq 0.2$) is generated uniformly and randomly. In the top subfigure, it can be observed that the state $x$ is eventually bounded in the presence of disturbances, and a larger bound of $w$ results in a larger ultimate bound of $x$; in the bottom subfigure, the input $u$ is piecewise-constant and it changes its value at each $t_k$ such that $\Gamma_x(e(t_k),x(t_k))\geq 0$. 
	\begin{figure}[!hb]
		\begin{center}
			\includegraphics[width=0.57\linewidth]{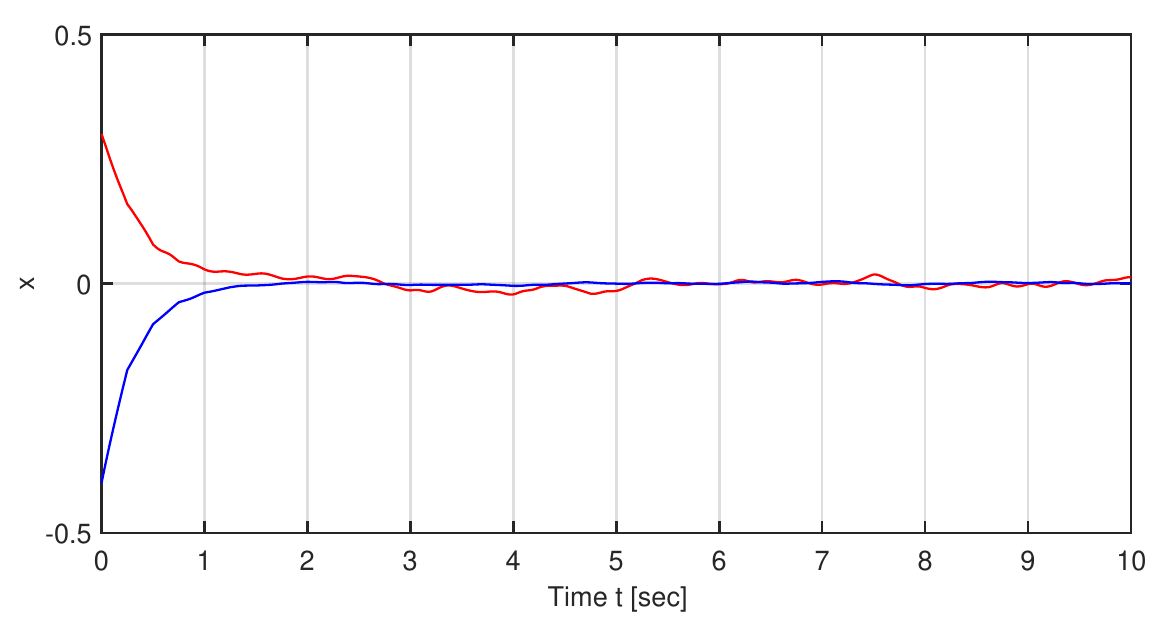}\vskip 1mm
			\includegraphics[width=0.57\linewidth]{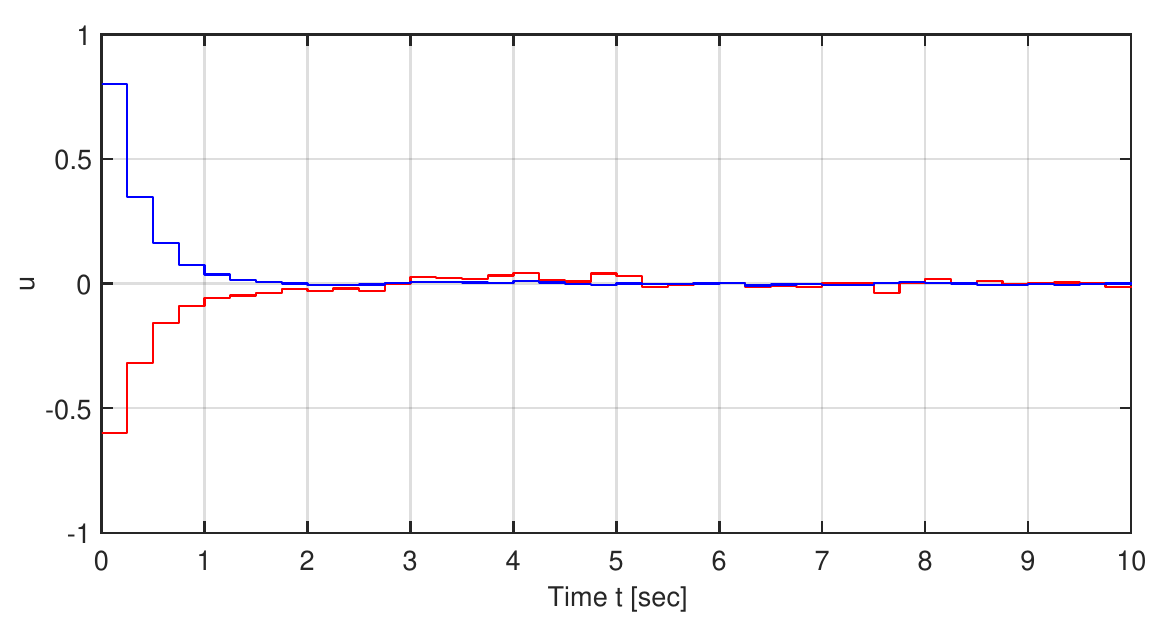}\vskip 1mm
			\includegraphics[width=0.56\linewidth]{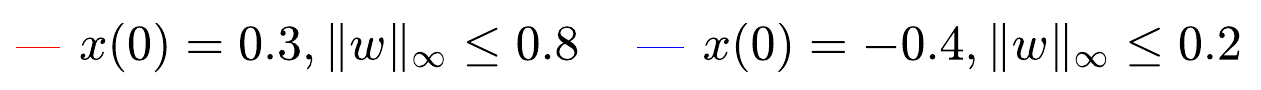}
			\caption{Trajectories of  $x,u$ in the simulation of Example \ref{ex1}.}\label{figex1}
		\end{center}
	\end{figure}

\end{example}


\begin{example}\label{ex2}
	Consider the following dynamical model of the  single-link robot arm given in \cite{abdelrahim2017robust}:
	\begin{equation*}
	\dot x_1=x_2,\;
	\dot x_2=-\sin(x_1)+u+w,\;
	y=x_1.
	\end{equation*}
	The system can be written in the form of \eqref{dyn1} with 
	$A\!=\!
	\begin{pmatrix}
	0&1\\
	0&0
	\end{pmatrix}$, $B\!=\!
	\begin{pmatrix}
	0\\
	1
	\end{pmatrix}$, $C\!=\!(1,0)$, $E\!=\!
	\begin{pmatrix}
	0\\
	-1
	\end{pmatrix}$, $E_w\!=\!
	\begin{pmatrix}
	0\\
	1
	\end{pmatrix}$, $q\!=\!x_1$, $C_q\!=\!(1,0)$, $p(q)\!=\!\sin(q)$.
	The nonlinearity $p$ is globally Lipschitz and satisfies the incremental quadratic constraint \eqref{eq:delQC} with $M=
	\begin{pmatrix}
	1&0\\
	0&-1
	\end{pmatrix}$. 
	%
	Consider the configuration in Fig.\ref{figstate} (b). 
	Assume that the continuous-time observer is \eqref{obserCT} and the control input is \eqref{inputCT}.  By the results of \cite{xuobserverarxiv18}, $K_1=(-7.3936,-3.9937)$,
	$K_2=1$, $L_1=-1$, $L_2=\begin{pmatrix}
	-5.1294\\-18.0352
	\end{pmatrix}$ can be chosen.  
	By letting $\alpha=1.1$, the LMI \eqref{LMI1} in Theorem \ref{thmoutput} is solved, which yields  the values of  $a_1,a_2,b_1,b_2,\mu_1,\mu_2,d,\gamma_1,\gamma_2,c_1,c_2$, from which  $\mathcal{T}(\mu_1,\gamma_1)=0.0751,\mathcal{T}(\mu_2,\gamma_2)=0.0639$. Then, solve the LMI \eqref{LMI2} to obtain the matrices $P_1=0.1462$ and $P_2=\begin{pmatrix}
	0.6307  &  0.1195 \\
	0.1195  &  0.1434
	\end{pmatrix}$.
	Choose $h=0.02,s=0.02,\lambda_1=0.627,\lambda_2=0.575,\alpha_0=1$. By Theorem \ref{thmoutput}, the triggering functions are chosen as $\Gamma_y(y,y_e)=0.9554\|y_e\|^2-0.02y^\top P_1y,
	\Gamma_u(\hat x,x_e)=1.1526\|x_e\|^2-0.02\hat x^\top P_2\hat x.$
	Choose the initial state as $x_1(0)=-0.2,x_2(0)=0.6,\hat x_1(0)=-0.3,\hat x_2(0)=0.7$, and let the disturbance be randomly generated and satisfies $\|w\|_\infty\leq 0.05$. The simulation results are shown in Fig. \ref{figex22}, where the trajectories of $x$, $\hat e$ and $u$ are plotted. It can be seen that $x_1,x_2,\hat e$  all eventually go to a neighborhood of the origin in the presence of disturbances.
		\begin{figure}[!ht]
		\begin{center}
			\includegraphics[width=0.6\linewidth]{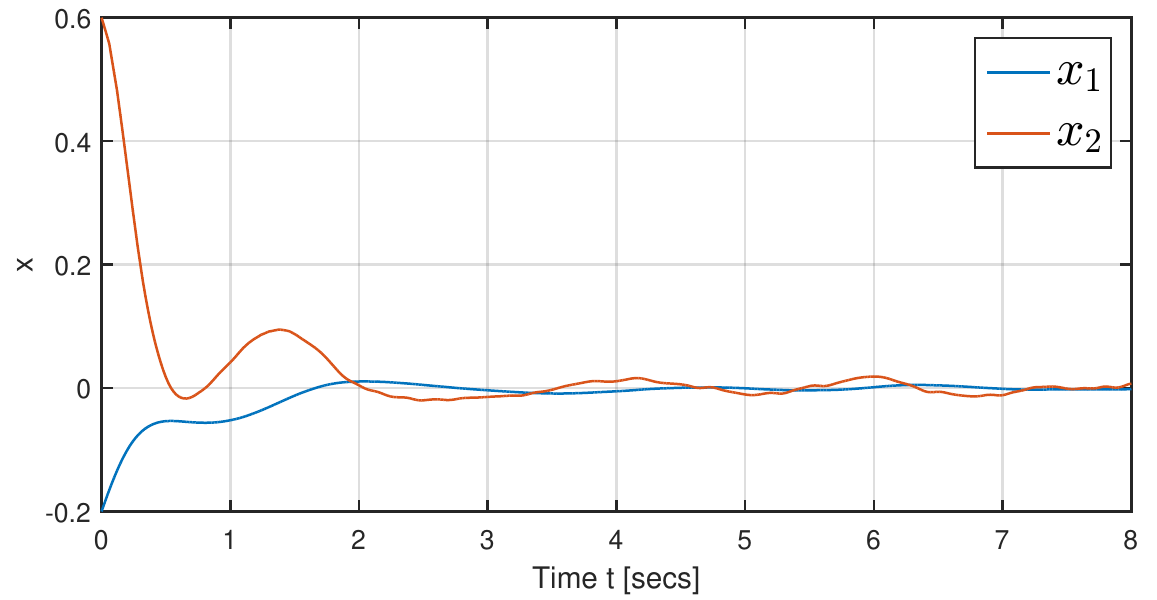}\vskip 1mm
			\includegraphics[width=0.6\linewidth]{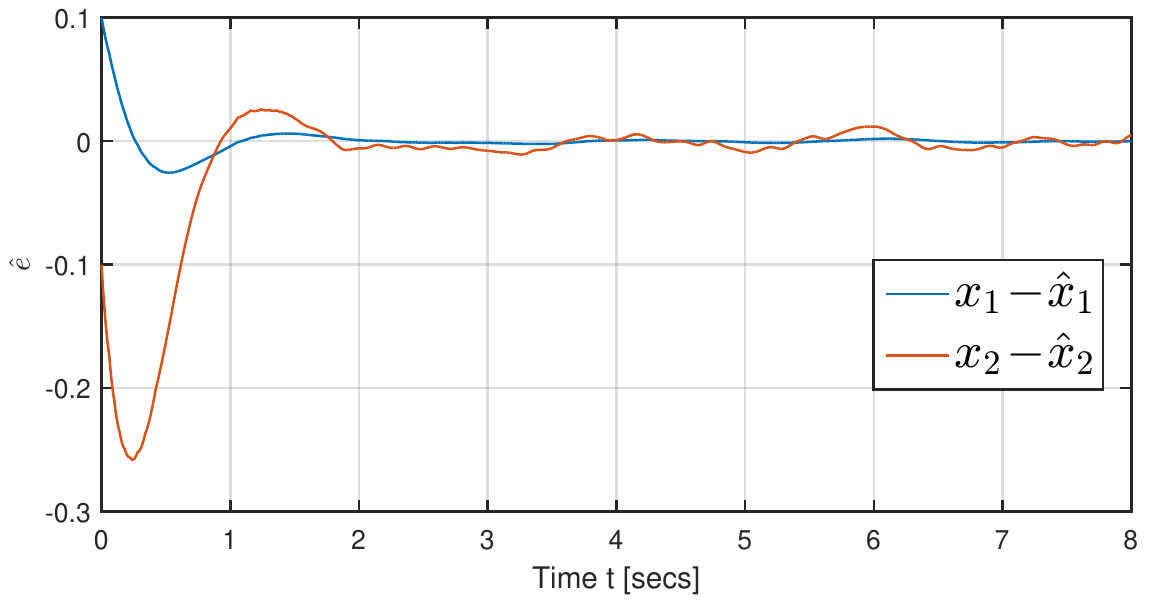}\vskip 1mm
			\includegraphics[width=0.6\linewidth]{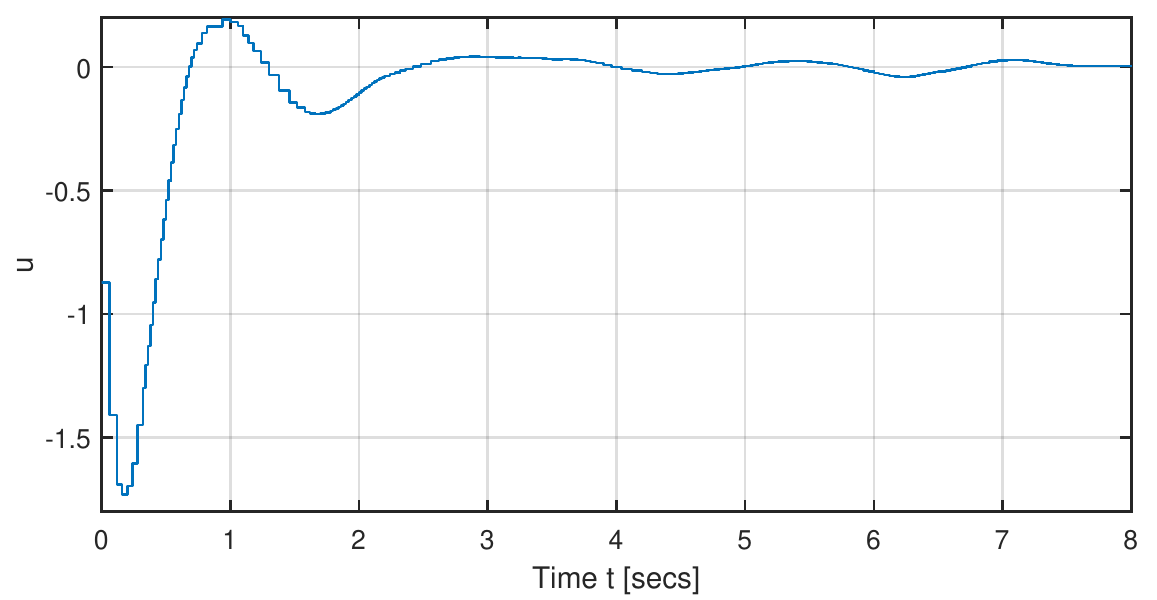}
			\caption{Trajectories of $x,\hat x,u$ in the simulation of Example \ref{ex2}. }\label{figex22}
		\end{center}
	\end{figure}

\end{example}

\section{Conclusion}\label{sec:conclusion}
This paper investigated periodic event-triggered control design for nonlinear systems subject to disturbances using the impulsive system approach. Sufficient conditions were proposed to ensure the resulting closed-loop system input-to-state stable using state feedback and observer-based output feedback controllers, respectively. LMI-based sufficient conditions for the PETC design of incrementally quadratic nonlinear systems were also proposed.  For all the cases considered, the sampling period and the triggering functions were given explicitly.


\bibliographystyle{IEEEtran}
\bibliography{./PETC}

\begin{thebibliography}{10}
\providecommand{\url}[1]{#1}
\csname url@samestyle\endcsname
\providecommand{\newblock}{\relax}
\providecommand{\bibinfo}[2]{#2}
\providecommand{\BIBentrySTDinterwordspacing}{\spaceskip=0pt\relax}
\providecommand{\BIBentryALTinterwordstretchfactor}{4}
\providecommand{\BIBentryALTinterwordspacing}{\spaceskip=\fontdimen2\font plus
\BIBentryALTinterwordstretchfactor\fontdimen3\font minus
  \fontdimen4\font\relax}
\providecommand{\BIBforeignlanguage}[2]{{%
\expandafter\ifx\csname l@#1\endcsname\relax
\typeout{** WARNING: IEEEtran.bst: No hyphenation pattern has been}%
\typeout{** loaded for the language `#1'. Using the pattern for}%
\typeout{** the default language instead.}%
\else
\language=\csname l@#1\endcsname
\fi
#2}}
\providecommand{\BIBdecl}{\relax}
\BIBdecl

\bibitem{heemels2012introduction}
W.~Heemels, K.~H. Johansson, and P.~Tabuada, ``An introduction to
  event-triggered and self-triggered control,'' in \emph{IEEE Conf. on Decision
  and Control}, 2012, pp. 3270--3285.

\bibitem{tabuada2007event}
P.~Tabuada, ``Event-triggered real-time scheduling of stabilizing control
  tasks,'' \emph{IEEE Transactions on Automatic Control}, vol.~52, no.~9, pp.
  1680--1685, 2007.

\bibitem{shoukry2016event}
Y.~Shoukry and P.~Tabuada, ``Event-triggered state observers for sparse sensor
  noise/attacks,'' \emph{IEEE Transactions on Automatic Control}, vol.~61,
  no.~8, pp. 2079--2091, 2016.

\bibitem{wang2011event}
X.~Wang and M.~D. Lemmon, ``Event-triggering in distributed networked control
  systems,'' \emph{IEEE Transactions on Automatic Control}, vol.~56, no.~3, pp.
  586--601, 2011.

\bibitem{garcia2013model}
E.~Garcia and P.~J. Antsaklis, ``Model-based event-triggered control for
  systems with quantization and time-varying network delays,'' \emph{IEEE
  Transactions on Automatic Control}, vol.~58, no.~2, pp. 422--434, 2013.

\bibitem{borgers2014event}
D.~Borgers and W.~M. Heemels, ``Event-separation properties of event-triggered
  control systems,'' \emph{IEEE Transactions on Automatic Control}, vol.~59,
  no.~10, pp. 2644--2656, 2014.

\bibitem{donkers2012output}
M.~Donkers and W.~Heemels, ``Output-based event-triggered control with
  guaranteed $\mathcal{L}_\infty$-gain and improved and decentralized
  event-triggering,'' \emph{IEEE Transactions on Automatic Control}, vol.~57,
  no.~6, pp. 1362--1376, 2012.

\bibitem{abdelrahim2017robust}
M.~Abdelrahim, R.~Postoyan, J.~Daafouz, and D.~Ne{\v{s}}i{\'c}, ``Robust
  event-triggered output feedback controllers for nonlinear systems,''
  \emph{Automatica}, vol.~75, pp. 96--108, 2017.

\bibitem{abdelrahim2016stabilization}
------, ``Stabilization of nonlinear systems using event-triggered output
  feedback controllers,'' \emph{IEEE Transactions on Automatic Control},
  vol.~61, no.~9, pp. 2682--2687, 2016.

\bibitem{heemels2013periodic}
W.~H. Heemels, M.~Donkers, and A.~R. Teel, ``Periodic event-triggered control
  for linear systems,'' \emph{IEEE Transactions on Automatic Control}, vol.~58,
  no.~4, pp. 847--861, 2013.

\bibitem{heemels2013model}
W.~Heemels and M.~Donkers, ``Model-based periodic event-triggered control for
  linear systems,'' \emph{Automatica}, vol.~49, no.~3, pp. 698--711, 2013.

\bibitem{heemels2015periodic}
W.~Heemels, R.~Postoyan, M.~Donkers, A.~R. Teel, A.~Anta, P.~Tabuada, and
  D.~Ne{\v{s}}ic, ``Periodic event-triggered control,'' \emph{Event-based
  control and signal processing}, pp. 105--120, 2015.

\bibitem{eqtami2010event}
A.~Eqtami, D.~V. Dimarogonas, and K.~J. Kyriakopoulos, ``Event-triggered
  control for discrete-time systems,'' in \emph{American Control Conference},
  2010, pp. 4719--4724.

\bibitem{linsenmayer2019periodic}
S.~Linsenmayer, D.~V. Dimarogonas, and F.~Allg{\"o}wer, ``Periodic
  event-triggered control for networked control systems based on non-monotonic
  lyapunov functions,'' \emph{Automatica}, vol. 106, pp. 35--46, 2019.

\bibitem{luc2017periodic}
L.~Etienne, S.~Di~Gennaro, and J.-P. Barbot, ``Periodic event-triggered
  observation and control for nonlinear lipschitz systems using impulsive
  observers,'' \emph{International Journal Of Robust And Nonlinear Control},
  vol.~27, no.~18, pp. 4363--4380, 2017.

\bibitem{borgers2018periodic}
D.~Borgers, R.~Postoyan, A.~Anta, P.~Tabuada, D.~Ne{\v{s}}i{\'c}, and
  W.~Heemels, ``Periodic event-triggered control of nonlinear systems using
  overapproximation techniques,'' \emph{Automatica}, vol.~94, pp. 81--87, 2018.

\bibitem{wang2016stabilization}
W.~Wang, R.~Postoyan, D.~Ne{\v{s}}i{\'c}, and W.~M.~H. Heemels, ``Stabilization
  of nonlinear systems using state-feedback periodic event-triggered
  controllers,'' in \emph{IEEE Conf. on Decision and Control}, 2016, pp.
  6808--6813.

\bibitem{aranda2017design}
E.~Aranda-Escol{\'a}stico, M.~Abdelrahim, M.~Guinaldo, S.~Dormido, and
  W.~Heemels, ``Design of periodic event-triggered control for polynomial
  systems: a delay system approach,'' \emph{IFAC-PapersOnLine}, vol.~50, no.~1,
  pp. 7887--7892, 2017.

\bibitem{yang2018periodic}
J.~Yang, J.~Sun, W.~X. Zheng, and S.~Li, ``Periodic event-triggered robust
  output feedback control for nonlinear uncertain systems with time-varying
  disturbance,'' \emph{Automatica}, vol.~94, pp. 324--333, 2018.

\bibitem{wang2018periodic}
W.~Wang, R.~Postoyan, D.~Ne{\v{s}}si{\'c}, and W.~Heemels, ``Periodic
  event-triggered output feedback control of nonlinear systems,'' in \emph{IEEE
  Conf. on Decision and Control}, 2018, pp. 957--962.

\bibitem{wang2019periodic}
W.~Wang, R.~Postoyan, D.~Nesic, and W.~Heemels, ``Periodic event-triggered
  control for nonlinear networked control systems,'' \emph{IEEE Transactions on
  Automatic Control}, vol.~65, no.~2, pp. 620--635, 2020.

\bibitem{khalil02book}
H.~Khalil, \emph{Noninear Systems (3rd Edition)}.\hskip 1em plus 0.5em minus
  0.4em\relax Prentice Hall, 2002.

\bibitem{hespanha2008lyapunov}
J.~P. Hespanha, D.~Liberzon, and A.~R. Teel, ``Lyapunov conditions for
  input-to-state stability of impulsive systems,'' \emph{Automatica}, vol.~44,
  no.~11, pp. 2735--2744, 2008.

\bibitem{lin1995various}
Y.~Lin, E.~Sontag, and Y.~Wang, ``Various results concerning set input-to-state
  stability,'' in \emph{IEEE Conf. on Decision and Control}, 1995, pp.
  1330--1335.

\bibitem{dashkovskiy2013input}
S.~Dashkovskiy and A.~Mironchenko, ``Input-to-state stability of nonlinear
  impulsive systems,'' \emph{SIAM Journal on Control and Optimization},
  vol.~51, no.~3, pp. 1962--1987, 2013.

\bibitem{carnevale2007lyapunov}
D.~Carnevale, A.~R. Teel, and D.~Nesic, ``A {L}yapunov proof of an improved
  maximum allowable transfer interval for networked control systems,''
  \emph{IEEE Transactions on Automatic Control}, vol.~52, no.~5, pp. 892--897,
  2007.

\bibitem{nesic2009explicit}
D.~Nesi{\'c}, A.~Teel, and D.~Carnevale, ``Explicit computation of the sampling
  period in emulation of controllers for nonlinear sampled-data systems,''
  \emph{IEEE Transactions on Automatic Control}, vol.~54, no.~3, pp. 619--624,
  2009.

\bibitem{sontag2008iss}
E.~D. Sontag, ``Input to state stability: Basic concepts and results,'' in
  \emph{Nonlinear and optimal control theory}.\hskip 1em plus 0.5em minus
  0.4em\relax Springer, 2008, pp. 163--220.

\bibitem{accikmecse2011observers}
B.~A{\c{c}}{\i}kme{\c{s}}e and M.~Corless, ``Observers for systems with
  nonlinearities satisfying incremental quadratic constraints,''
  \emph{Automatica}, vol.~47, no.~7, pp. 1339--1348, 2011.

\bibitem{xuobserverarxiv18}
\BIBentryALTinterwordspacing
X.~Xu, B.~A{\c{c}}{\i}kme{\c{s}}e, and M.~Corless, ``Observer-based controllers
  for incrementally quadratic nonlinear systems with disturbances,'' \emph{IEEE
  Transactions on Automatic Control}, 2020. [Online]. Available:
  \url{https://ieeexplore.ieee.org/document/9099370}
\BIBentrySTDinterwordspacing

\bibitem{papachristodoulou2013sostools}
\BIBentryALTinterwordspacing
A.~Papachristodoulou, J.~Anderson, G.~Valmorbida, S.~Prajna, P.~Seiler, and
  P.~Parrilo, ``{SOSTOOLS} version 3.00 {S}um of {S}quares optimization toolbox
  for {M}atlab,'' \emph{arXiv:1310.4716}, 2013. [Online]. Available:
  \url{http://arxiv.org/abs/1310.4716}
\BIBentrySTDinterwordspacing

\end{thebibliography}

\end{document}